\documentclass[final,1p]
{elsarticle}

\usepackage{amsmath, amsthm, amssymb}
\usepackage{amsfonts}
\usepackage{amsmath}
\usepackage{amssymb}
\usepackage{array}
\usepackage{stmaryrd}
\usepackage{graphicx}
\usepackage{color}
\usepackage{tikz}

\usepackage{enumitem}

\newtheorem{theorem}{Theorem}[section]
\newtheorem{lemma}[theorem]{Lemma}

\newtheorem{corollary}[theorem]{Corollary}

\newtheorem{remark}[theorem]{\it Remark}
\newtheorem{theorema}{Theorem}

\numberwithin{equation}{section}

\newcommand{\pt}{\partial}
\newcommand{\D}{\partial}

\newcommand {\beq} {\begin{equation}}
\newcommand {\eeq} {\end{equation}}
\renewcommand{\sim}{\simeq}

\newcommand {\U} {{\mathcal U}}

\newcommand{\LL}{{\mathcal L}}

\newcommand{\R}{\mathbb{R}}

\definecolor{blue}{rgb}{0,0,0.9}
\definecolor{pass}{rgb}{0,0,0.6}

\begin{document}

\begin{frontmatter}



\title{
Error analysis of an L2-type method on graded meshes for semilinear subdiffusion equations}


\author{Natalia Kopteva\corref{cor1}}
\cortext[cor1]{Department of Mathematics and Statistics, University of Limerick, Limerick, Ireland}
\ead{natalia.kopteva@ul.ie}


\begin{abstract}
A semilinear initial-boundary value problem with a Caputo time derivative of fractional order $\alpha\in(0,1)$ is considered,
 solutions of which typically exhibit a singular behaviour at an initial time.
 For an L2-type discretization of order $3-\alpha$, we give sharp pointwise-in-time error bounds
on graded temporal meshes with arbitrary degree of grading.
\end{abstract}







\end{frontmatter}

\section{Introduction. Main result}
Our purpose 
is to extend
sharp pointwise-in-time error bounds obtained in \cite{NK_L2}
for L2-type discretizations of linear subdiffusion equations
to the semilinear 
case:
\beq\label{problem}
\begin{array}{l}
\D_t^{\alpha}u+\LL u+f(x,t,u)=0\quad\mbox{for}\;\;(x,t)\in\Omega\times(0,T],\\[0.2cm]
u(x,t)=0\quad\mbox{for}\;\;(x,t)\in\pt\Omega\times(0,T],\qquad
u(x,0)=u_0(x)\quad\mbox{for}\;\;x\in\Omega.
\end{array}
\eeq
This problem is posed in a bounded Lipschitz domain  $\Omega\subset\R^d$, $d\in\{1,2,3\}$.
The operator $\D_t^\alpha$, for some $\alpha\in(0,1)$, is
the Caputo fractional derivative in time defined \cite{Diet10} $\forall\,t>0$ by\vspace{-0.1cm}
$$
\D_t^{\alpha} u(\cdot,t) :=  \frac1{\Gamma(1-\alpha)} \int_{0}^t(t-s)^{-\alpha}\, \pt_s u(\cdot, s)\, ds,
$$
where $\Gamma(\cdot)$ is the Gamma function, and $\pt_s$ denotes the partial derivative in $s$.
The spatial operator $\LL$ in \eqref{problem} is a linear self-ajoint second-order elliptic operator defined by
$\LL u := \sum_{k=1}^d \bigl\{-\pt_{x_k}\!(a_k(x)\,\pt_{x_k}\!u)  \bigr\}$,
with sufficiently smooth coefficients $\{a_k\}$ 
in $C(\bar\Omega)$, 
which also satisfy $a_k>0$
in $\bar\Omega$.

Throughout the paper, we make the following assumption on $f$.\vspace{-0.1cm}

\begin{itemize}
\item[{\bf A1.}]
Let $f$ be continuous in $s$ and  satisfy $f(\cdot,t,s)\in L_\infty(\Omega)$ for all $t>0$ and $s\in\R$,
and the Lipschitz condition with some constant $\lambda \ge 0$:
$$
|f(x,t,s_1)-f(x,t,s_2)|\le \lambda|s_1-s_2|\quad \forall s_1, s_2\in \R,\;\; x\in\Omega,\;\;t>0.
$$
\end{itemize}\vspace{-0.3cm}

\begin{remark}[A1 within the exact solution range]
While, strictly speaking, Allen-Cahn-type equations (with the cubic $f=u^3-u$)
and Fisher-type equations (with the quadratic $f=u^2-u$), do not satisfy A1, our error analysis still applies, as described in \cite[{\S8.1}]{NK_semil_L1}.
\end{remark}

We shall focus on  the semidiscretization of our problem~\eqref{problem} in time of type
\beq\label{semediscr_method}
\delta_t^\alpha U^m +\LL U^m+ f(\cdot,t_m, U^m)=0\;\;\mbox{in}\;\Omega,\quad U^m=0\;\;\mbox{on}\;\pt\Omega\quad\forall\,m=1,\ldots,M,
\eeq
subject to $U^0=u_0$,
associated with the temporal mesh $0=t_0<t_1<\ldots <t_M=T$. 
(An extension to the fully discrete case can be carried out in lines
with \cite[{\S5.2}]{NK_L2}.)

The discrete fractional-derivative operator $\delta_t^\alpha$ in \eqref{semediscr_method} is
an L2-type
discrete fractional-derivative operator for $\pt_t^{\alpha}$ 
\cite{higher_order,NK_L2,quan_wu_L2_sinum23},
 based on piecewise-quadratic Lagrange interpolants:
\begin{subequations}\label{delta_t_def}
\beq
\delta_{t}^{\alpha} U^m :=\D^\alpha_t (\Pi^mU)(t_m),\quad
\quad
\Pi^m:=\left\{\begin{array}{cll}
\Pi_{1,1}&\mbox{on~}(0,t_1)&\mbox{for~}m=1,\\
\Pi_{2,j}&\mbox{on~}(t_{j-1},t_j)&\mbox{for~}1\le j<m,\\
\Pi_{2,j-1}&\mbox{on~}(t_{j-1},t_j)\;&\mbox{for~}j=m>1,\\
\end{array}\right.
\eeq
where
 $\Pi_{1,j}$ and $\Pi_{2,j}$ are the standard linear and quadratic Lagrange interpolation operators with the following
 interpolation points:
\beq
\Pi_{1,j}\;:\;\{t_{j-1},t_j\},\qquad\qquad
\Pi_{2,j}\;:\;\{t_{j-1},t_j,t_{j+1}\}.
\eeq
\end{subequations}
The above method \eqref{semediscr_method},\,\eqref{delta_t_def} is of order $3-\alpha$ in time on appropriate temporal meshes (by comparison, while the Alikhanov discretization of $\pt_t^\alpha$ also enjoys the order $3-\alpha$, in the parabolic case it deteriorates to order $2$; see, e.g., \cite[Theorem~4.4 v Theorem~4.10]{NK_XM}).


There is a substantial literature on the L1 \cite{stynes_etal_sinum17,sinum18_liao_et_al,sinum19_liao_et_al,NK_MC_L1,NK_XM} and Alikhanov \cite{ChenMS_JSC,NK_XM,JCP_2021_liao_et_al}  schemes in the context of graded/nonuniform temporal meshes,
the analyses relying on that the discrete time-fractional operators are associated with Z-matrices (i.e. matrices with non-positive off-diagonal values); the semilinear case has also been addressed
\cite{Ji_Liao_L1,NK_semil_L1,Ji_Liao_Al}. 

Far fewer results are available for L2-type discretizations, such as \eqref{delta_t_def}, the difficulty in the analysis, to a large degree, due to that the operator $\delta_t^\alpha$ is not associated with a Z-matrix even on uniform meshes \cite{higher_order,NK_L2,quan_wu_L2_sinum23}
(for similar L2-type schemes, see also \cite{jcp_gao,jcp_xing_yan}).
 In \cite{higher_order} this operator
is analysed  on uniform temporal meshes, and the optimal convergence order
 $3-\alpha$ in time is established
under unrealistically strong regularity assumptions on the exact solution.
The case of nonuniform temporal meshes was addressed in \cite{NK_L2}  building on the framework, developed in \cite{NK_MC_L1,NK_XM}, based on monotonicity arguments and discrete barrier functions.  An alternative  approach was recently presented in \cite{quan_wu_L2_sinum23}.

Importantly, in extending the framework of \cite{NK_MC_L1,NK_XM} to the linear version of \eqref{semediscr_method},\,\eqref{delta_t_def} in \cite{NK_L2},  the key was the following
inverse-monotone representation for $\delta_t^\alpha$:\vspace{-0.1cm}
\begin{subequations}\label{UV_ab}
\beq\label{UV}
\delta_t^\alpha U^m=\sum_{j=0}^m \kappa_{m,j}V^j\;\;\forall\,m\ge 1,
\quad V^j:=\frac{U^j-\beta_j U^{j-1}}{1-\beta_j}\;\;\forall\,j\ge 1,\quad
V^0:=U^0,\vspace{-0.1cm}
\eeq
where $\beta_j\in [0,1)$.
Then, under certain conditions on the mesh,
 a set $\{\beta_j\}$ was chosen such that\vspace{-0.3cm}
\beq\label{UV_kappa}
\kappa_{m,m}>0\;\;\mbox{and}\;\;\sum_{j=0}^{m}\kappa_{m,j}=0\;\;\forall\,m\ge 1,\quad \kappa_{m,j}\le 0\;\;\forall\,0\le j<m\le M.\vspace{-0.3cm}
\eeq
\end{subequations}%
By \eqref{UV_ab}, the matrix associated with $\delta_t^\alpha$ is immediately inverse-monotone, i.e. all elements of the inverse to
this matrix are non-negative (as it is a product of two M-matrices, which are inverse-monotone). The latter property is equivalent to a discrete comparison principle (as in \cite[Theorem~3.1]{NK_L2} or, equivalently,
Lemma~\ref{lem_comp} for $\lambda=0$), which plays a key role in the analysis of \cite{NK_L2}.

{\it Novelty.} In this paper, an extension of the sharp pointwise-in-time error bounds of \cite{NK_L2} from the linear to the semilinear case generally follows the approach of \cite{NK_semil_L1}, but it
hinges on another nontrivial result, the inverse-monotonicity of the discrete operator $\delta_t^\alpha-\lambda$, where $\lambda>0$ is from assumption A1 on $f$
(see Lemma~\ref{lem_comp}). While elementary, our proof of this result is novel, and, importantly, it immediately applies to any inverse-monotone discretization of the Caputo derivative $\pt_t^\alpha$.
Another rather surprising feature of the L2 operator $\delta_t^\alpha$ of \eqref{delta_t_def} is
worth noting.
The inverse-monotonicity of $\delta_t^\alpha+\lambda$ (for $\lambda>0$) remains unclear,
while one might expect this to immediately follow from the inverse monotonicity of $\delta_t^\alpha$ (this feature is in sharp contrast to
the Caputo derivative operator $\pt_t^\alpha$ or any Z-matrix discretization of the latter, such as the L1 or Alikhanov operators).
\medskip

{\it Notation.}
We write
 $a\sim b$ when $a \lesssim b$ and $a \gtrsim b$, and
$a \lesssim b$ when $a \le Cb$ with a generic constant $C$ depending on $\Omega$, $T$, $u_0$,
$f$, and $\alpha$,
but not
 on the total number of degrees of freedom.
 $H^1_0(\Omega)$ is the standard
  space of functions in the Sobolev space $W^1_2(\Omega)$ vanishing on $\pt\Omega$.
Additionally, the following main result of the paper and further assumptions below involve
$\tau_j:=t_j-t_{j-1}$ and
$\rho_j:=\frac{\tau_{j}}{\tau_{j-1}}$.

\begin{theorem}\label{theo_semi}
Let the graded temporal mesh
$\{t_j=T(j/M)^r\}_{j=0}^M$,
with any fixed $r\ge1$, satisfy
$\lambda\tau_j^{\alpha}< \{\Gamma(2-\alpha)\}^{-1}$ $\forall\,j\ge1$
Suppose that $u$ is a unique solution of \eqref{problem} with the initial data $u_0\in L_\infty(\Omega)$ and
 under assumption {\rm A1} on $f$.
 Also, for $t\in(0,T]$
suppose that
$u(\cdot,t)\in H^1_0(\Omega)$
and
$\|\partial_t^l u (\cdot, t)\|_{L_2(\Omega)}\lesssim 1+t^{\alpha-l}$, $l = 1,3$.
Then
there exists a unique solution   $\{U^m\in H^1_0(\Omega)\cap L_{\infty}(\Omega)\}_{m=1}^M$ of \eqref{semediscr_method}, for which one has
\beq\label{L2_semi_error}
\|u(\cdot,t_m)-U^m\|_{L_2(\Omega)}\lesssim {\mathcal E}^m:=
\left\{\begin{array}{ll}
M^{-r}\,t_m^{\alpha-1}&\mbox{if~}1\le r<3-\alpha,\\[0.2cm]
M^{\alpha-3}\,t_m^{\alpha-1}[1+\ln(t_m/t_1)]&\mbox{if~}r=3-\alpha.\\[0.2cm]
M^{\alpha-3}\,t_m^{\alpha-(3-\alpha)/r}
&\mbox{if~}r>3-\alpha.
\end{array}\right.
\eeq
\end{theorem}
While the above theorem is formulated for a practical case of the standard graded mesh, it applies to more general temporal meshes that satisfy the following assumptions A2 and A3
(which is reflected in a more general statement of Theorem~\ref{theo_star} in \S\ref{sec_proof}).%

\begin{itemize}
\item[{\bf A2\,\,}]
Assume that
$1\le\rho_{j+1}\le\rho_{j}\le \bar\rho^*$
 $\forall\, j\ge K+1$, for some integer $1\le K\lesssim 1$ (i.e. $K$ is sufficiently large, but independent of $M$),
where
$\bar\rho^*:=2/(1-\bar\sigma^*)-1$, with a  constant
$\bar\sigma^*=\bar\sigma^*(\alpha)$ in $(0,1)$
 from \cite[Lemma~5.1]{NK_L2}
(for the existence and computation of which see also \cite[Corollary~2.10 and Remark~2.12]{NK_L2}).%
\vspace{-0.1cm}
%

\item[{\bf A2${}^{*\,}\!$}]
Assumption A2 is satisfied with $K=1$.\vspace{-0.1cm}

\item[{\bf A3\,\,}] Given $\gamma\in\R$, let the temporal mesh satisfy\vspace{-0.1cm}
$$
\tau _1\simeq M^{-r},\qquad 
\tau_j
\simeq t_j/j,
\qquad 
t_j \sim \tau_1 j^r
\qquad
\forall\,j=1,\ldots,M,\vspace{-0.1cm}
$$
for some $1\le r\le (3-\alpha)/\alpha$
if $\gamma>\alpha-1$ or for some $r\ge1$ if $\gamma\le \alpha-1$.
\end{itemize}

\begin{remark}\label{rem_graded_A23}
For any fixed $\alpha\in(0,1)$ and $r\ge 1$, there exists $K=K(\alpha,r)\lesssim 1$ such that the graded temporal mesh $\{t_j=T(j/M)^r\}_{j=0}^M$ satisfies A2 and A3; see \cite[Corollary~3.3]{NK_L2}.
\end{remark}

\begin{remark}[regularity of $u_0$] \color{blue}It is assumed in the above Theorem~\ref{theo_semi} that
the initial data $u_0\in L_\infty(\Omega)$. By comparison, \cite{Jin_Li_Zhou} addresses
a version of \eqref{problem} under the assumption $u_0\in H_0^1(\Omega)\cap H^2(\Omega)$ (with solution regularity bounds given by \cite[Theorem~3.1]{Jin_Li_Zhou}),
which, for smooth and convex polyhedral domains, implies
$u_0\in L_\infty(\Omega)$. Note also that when assuming $\|\partial_t^l u (\cdot, t)\|_{L_2(\Omega)}\lesssim 1+t^{\alpha-l}$, we indirectly impose $u_0\in H_0^1(\Omega)\cap H^2(\Omega)$
even in the linear case.
\end{remark}


\section{Discrete comparison principle and stability properties for $\delta_t^\alpha-\lambda$}\label{sec_stab}

\begin{remark}[matrix associated with $\delta_t^\alpha$]\label{rem_A}
One can easily see that the operator $\delta_t^\alpha$ of \eqref{delta_t_def} allows a representation $\delta_t^\alpha U^m=\sum_{j\le m}a_{mj}U^j$ $\forall\,m\ge 1$. Hence, it is
associated with a lower triangular matrix $A$ with elements $\{a_{mj}\}$.
While the signs of its off-diagonal elements vary,
the diagonal elements $a_{mm}=\D_t^\alpha\Pi_{2,m-1}\phi^m(t)$ for $m>1$
and $a_{mm}=\D_t^\alpha\Pi_{1,1}\phi^1(t_1)$ for $m=1$, where $\phi^m$ is the standard basis function equal to $1$ at $t_m$ and vanishing at $t_j\neq t_m$.
Hence, a calculation (using integration by parts combined with
$[\Pi_{2,m-1}-\Pi_{1,m}]\phi^m(t)\le 0$ for $t\le t_m$) yields $\D_t^\alpha[\Pi_{2,m-1}-\Pi_{1,m}]\phi^m(t_m)\ge 0$. Hence,  $a_{mm}\ge \D_t^\alpha\Pi_{1,m}\phi^m(t_m)$,
or, equivalently, $a_{mm}\ge  \tau_m^{-\alpha}\{\Gamma(2-\alpha)\}^{-1}$
  $\forall\, m\ge 1$.
  \end{remark}

  \begin{lemma}[Comparison principle for $\delta_t^\alpha-\lambda$]\label{lem_comp}
Suppose that
the temporal grid $\{t_j\}_{j=0}^M$ satisfies A2${}^*$ and
$\lambda\tau_j^{\alpha}< \{\Gamma(2-\alpha)\}^{-1}$ $\forall\,j\ge1$.
Then $U^0\le B^0$ and  $(\delta_t^\alpha-\lambda)U^m\le (\delta_t^\alpha-\lambda)B^m$ $\forall\,m\ge1$
imply that $U^m\le B^m$ $\forall\,m\ge0$.
\end{lemma}

\begin{proof}
As the operator $\delta_t^\alpha-\lambda$ is linear, it suffices to consider the case $B^j=0$
$\forall\,j$, i.e. to show that $U^0\le 0$ and  $(\delta_t^\alpha-\lambda)U^m\le 0$ $\forall\,m\ge1$
imply $U^m\le 0$ $\forall\,m\ge0$.
Thus, we have $A\vec{U}-\lambda \vec{U}\le 0$,
where $A$ is an $(M+1)\times (M+1)$ matrix associated with $\delta_t^\alpha$, while the notation of type
 $\vec{U}:=\{U^j\}_{j=0}^M$ is used for the corresponding column vector.
Note that, by \cite[Corollary~2.10]{NK_L2} (in which a milder upper bound on $\sigma_j=1-2/(1+\rho_j)$,
and so on $\rho_j$, is imposed
compared to
\cite[Lemma~5.1]{NK_L2}), assumption A2${}^*$ is sufficient
for the inverse-monotone representation \eqref{UV_ab}.
Hence, the matrix $A$  is inverse-monotone, i.e. the elements of its inverse, denoted by $a_{mj}^*$, are non-negative
(see also \cite[Remark~2.1]{NK_L2}.
Now $A\vec{U}\le \lambda \vec{U}$ implies $\vec{U}\le \lambda A^{-1}\vec{U}$, or, equivalently,
$(1-\lambda a^*_{mm})U^m\le \lambda\sum_{j<m}a^*_{mj} U^j$,
as $A^{-1}$ is lower triangular, as well as $A$.
Note that here $1-\lambda a^*_{mm}>0 $, in view of $a^*_{mm}=1/a_{mm}$
combined with $\lambda\tau_m^{\alpha}< \{\Gamma(2-\alpha)\}^{-1}$ and Remark~\ref{rem_A}.
The desired assertion follows by induction.
\end{proof}

\begin{theorem}[Stability of $\delta_t^\alpha-\lambda$]\label{theo_main_stab_semi}
Given $\gamma\in\R$,
let  the temporal mesh 
satisfy A2${}^*$, A3,  and
$\lambda\tau_j^{\alpha}< \{\Gamma(2-\alpha)\}^{-1}$ $\forall\,j\ge1$.
Then 
for $\{U^j\}_{j=0}^M$
 one has
\beq\label{main_stab_new}
\left.\!\!\!\!\begin{array}{c}
(\delta_t^\alpha-\lambda) U^j\lesssim (\tau_1/ t_j)^{\gamma+1}
\\[0.2cm]
\forall j\ge1,\;\;\; U^0=0
\end{array}\hspace{-0.15cm}\right\}
\; \Rightarrow\;
U^j\lesssim
\U^j(\tau_1;\gamma)
:=
\ell_\gamma\tau_1 t_j^{\alpha-1}(\tau_1/t_j)^{\min\{0,\,\gamma\}}\;\;
\forall j\ge 1,
\eeq
where $\ell_\gamma:=1+\ln(T/\tau_1)$ for $\gamma=0$ and $\ell_\gamma:=1$ otherwise.
\end{theorem}

\begin{corollary}\label{cor_main_stab_XM_star}
Under the conditions of Theorem~\ref{theo_main_stab_semi} on the temporal mesh,
the operator $\delta_t^\alpha$ allows a representation of type \eqref{UV_ab}.
Furthermore,
for $\{U^j\}_{j=0}^M$ and $\{W^j\}_{j=0}^M$ with
 $U^0=W^0=0$ the following is true:
\beq\label{star_eq}
\left.
\begin{array}{cc}
\displaystyle\sum_{j=0}^m \kappa_{m,j}W^j-\lambda U^m\lesssim(\tau_1/ t_m)^{\gamma+1}&\forall\,m\ge 1\\
\displaystyle
\frac{U^j-\beta_j U^{j-1}}{1-\beta_j}\le W^j
&\forall\,j\ge 1
\end{array}\right\}
\quad\Rightarrow\quad
U^j\lesssim
\U^j(\tau_1;\gamma),
\eeq
where
the coefficients $\{\kappa_{m,j}\}$ and $\{\beta_j\}$ are from \eqref{UV_ab}, and
$\U^j$ is defined in \eqref{main_stab_new}.
\end{corollary}

\noindent{\it Proof of Theorem~\ref{theo_main_stab_semi}.}
Note that for $\lambda=0$ the desired assertion immediately follows from
\cite[Theorem~3.2]{NK_L2},
with the assumptions of the latter satisfied in view of A2${}^*$ and A3.
To extend \eqref{main_stab_new} from the case $\lambda=0$ to $\lambda>0$ and sufficiently large $M$ (otherwise, for $M\lesssim 1$, the desired assertion is straightforward),
in view of Lemma~\ref{lem_comp},
it suffices to construct an appropriate barrier function.
This function is constructed exactly as in the proof of
\cite[Theorem~3.1]{NK_semil_L1} with minimal changes (the discrete comparison principle of Lemma~\ref{lem_comp} playing a key role in this proof).

Importantly, the special case of $\gamma=0$ and $\lambda>0$
(which involves $\ell_\gamma:=1+\ln(T/\tau_1)$)
was not addressed in \cite[Theorem~3.1]{NK_semil_L1}.
To include this case, one needs the following minor modifications
in the proof of
\cite[Theorem~3.1(i)]{NK_semil_L1}. With the notation of the latter proof,
the definition of  the auxiliary barrier ${\mathcal B}^j_\gamma$ for any $\gamma\in\R$
should be modified to
$\delta_t^\alpha {\mathcal B}_\gamma^j= \ell_\gamma^{-1}(\tau/ t_j)^{1+\gamma}$
$\forall j\ge1$ subject to  ${\mathcal B}_\gamma^0=0$. Then the construction of the barrier ${\mathcal W}^j$
remains unchanged,
which leads to $U^j\lesssim \ell_\gamma{\mathcal W}^j$, equivalent to our \eqref{main_stab_new}.%
\hfill$\square$

\begin{remark}[{sharper $\ell_\gamma$}]\label{rem_ell_sharp}
For $\gamma=0$,
a sharper version of \eqref{main_stab_new} holds true with
 $\ell_\gamma=\ell_0:=1+\ln(t_j/\tau_1)$ factor in $\U^j$.
Indeed, an inspection of the above proof shows that \eqref{main_stab_new} is valid with unchanged constants for any $T\ge t_1$ (as long as $T\lesssim 1$). Hence, replacing $T$ by $t_m$, one gets
$U^j\lesssim [1+\ln(t_m/\tau_1)]\tau_1 t_j^{\alpha-1}$ $\forall\,j\le m$,
so
$U^m\lesssim [1+\ln(t_m/\tau_1)]\tau_1 t_m^{\alpha-1}$ $\forall\,m$.
\end{remark}

\begin{remark}[{generalization for the L1 method
}]
A version of \eqref{main_stab_new}
for the L1 discrete operator in place of $\delta_t^\alpha$
is given by
\cite[Theorem~3.1]{NK_semil_L1}, however, with
the case of $\gamma=0$ and $\lambda>0$ excluded.
Since the relevant part in the above proof also applies to the L1 operator,
we  now conclude that \cite[Theorem~3.1]{NK_semil_L1} is valid with
\cite[(1.5)]{NK_semil_L1} replaced by its generalized version~\eqref{main_stab_new}
(which is different in that it now includes the factor $\ell_\gamma$ and, hence, applies 
whether $\lambda=0$ or $\lambda>0$).
Hence, we immediately obtain a sharper version of the error bound \cite[(4.2)]{NK_semil_L1}, with 
for $r=2-\alpha$
from ${\mathcal E}^m= M^{(\alpha-2)(1-\epsilon)}\,t_m^{\alpha-(1-\epsilon)}$
(with an arbitrarily small positive $\epsilon$)
to  ${\mathcal E}^m= \ell_0 M^{\alpha-2}\,t_m^{\alpha-1}$.
A similar improvement also applies to all other error bounds in \cite{NK_semil_L1} which involve ${\mathcal E}^m$.
\end{remark}

\noindent{\it Proof of Corollary~\ref{cor_main_stab_XM_star}.} First
First, recall
from the proof of Lemma~\ref{lem_comp}
that A2${}^*$ is, indeed, sufficient for the inverse-monotone representation \eqref{UV_ab}.
Now, note that the assumptions in \eqref{theo_main_stab_semi}, rewritten in the matrix form, become
$A_1 \vec{W}- \lambda \vec{U}\lesssim \vec{F}$
and $A_2\vec{U}\le \vec W$,
with   lower triangular
 $M\times M$ matrices $A_1$ and $A_2$, and the notation of type
 $\vec{U}:=\{U^j\}_{j=1}^M$  used for the corresponding column vectors.
 For $\vec{F}$, we use
$F^m:=(\tau_1/ t_m)^{\gamma+1}$  for $m\ge1$.
Note that being M-matrices (i.e. diagonally dominant, with non-positive off-diagonal elements,
which follows from \eqref{UV_ab}),
both $A_1$ and $A_2$ are inverse-monotone.
To complete the proof, define the auxiliary $\vec{U}^*$ by
$A_2 \vec{U}^*=\vec{W}$, which immediately implies that $\vec{U}\le\vec{U}^*$ elementwise
(as $A_2$ is inverse-monotone).
Hence, one gets
$$
(A_1A_2-\lambda)\vec{U}^*=A_1 \vec{W}-\lambda\vec{U}^*\le A_1 \vec{W}-\lambda U \lesssim  \vec{F},
$$
or, equivalently, $(\delta_t^\alpha -\lambda)U^{*,m}\lesssim F^m$ subject to $U^{*,0}=0$.
So, an application of Theorem~\ref{theo_main_stab_semi} yields the
desired bound $U^j\le U^{*,j}\lesssim \U^j$.
\hfill$\square$

\section{Error analysis. Proof of Theorem~\ref{theo_semi}}\label{sec_proof}

In view of Remark~\ref{rem_graded_A23}, it suffices to prove Theorem~\ref{theo_semi}
under more general assumptions A2 and A3; see Theorem~\ref{theo_star} below.

\begin{lemma}[Stability for parabolic case]\label{lem_semi_stab}
Given $\gamma\in\R$,
let 
the temporal mesh 
satisfy A2, A3,   and
$\lambda\tau_j^{\alpha}< \{\Gamma(2-\alpha)\}^{-1}$ $\forall\,j\ge1$.
Then for $\{U^j\}_{j=0}^M$ from \eqref{semediscr_method} one has
\beq\label{semi_stab}
 %
 \left.\begin{array}{c}
\|f(\cdot,t_j, U^j)\|_{L_2(\Omega)}-\lambda \| U^j\|_{L_2(\Omega)}
 \lesssim (\tau_1/ t_j)^{\gamma+1}
\\[0.2cm]
\forall j\ge1,\;\;\; U^0=0\;\;\mbox{in}\;\bar\Omega
\end{array}\right\}
 \;\;\Rightarrow\;\;
 \|U^j\|_{L_2(\Omega)}\lesssim
\U^j(\tau_1;\gamma)\;\;\forall\, j\ge 1,
\eeq
where $\U^j$ is defined in \eqref{main_stab_new}.
\end{lemma}

\begin{proof}
(i)
We shall start by proving the desired assertion \eqref{semi_stab}
under the simplifying assumptions that, in addition to $U^0=0$, we also have
$U^1=0$, and
a more restrictive condition A2${}^*$ is satisfied instead of A2
(or, equivalently, we have A2 with $K=1$).
Not only A2${}^*$ enables us to employ
the results of section~\ref{sec_stab}, including
Corollary~\ref{cor_main_stab_XM_star}.
It also implies that the assumptions on the temporal grid made in
\cite[Lemma~5.1]{NK_semil_L1} are satisfied, so we can imitate
part~(ii) of the proof of this lemma.

To be more precise,
set
$V^0=0$ and
$V^m=\frac{1}{1-\beta_m}U^m-\frac{\beta_m}{1-\beta_m}U^{m-1}$ for $m\ge 1$,
and also
$$
W^m:=\sqrt{\displaystyle\|V^m\|_{L_2(\Omega)}^2+
\kappa_m^*\,
\langle\LL U^m, U^m\rangle},
\qquad\mbox{where}\quad \kappa^*_m:={\textstyle\frac{\kappa_{m,m}^{-1}}{1-\beta_m}}\,.
$$
Now, on representing $\delta_t^\alpha U^m$ in \eqref{semediscr_method} in terms of $\{V^j\}$
in view of \eqref{UV},
take the $L_2(\Omega)$ inner product
$\langle\cdot,\cdot\rangle$ of \eqref{semediscr_method} with
$V^m$. With the simplified notation $\|\cdot\|:=\|\cdot\|_{L_2(\Omega)}$,
note that $\langle\delta_t^\alpha U^m,V^m\rangle=
\sum_{j=0}^m \kappa_{m,j}\langle V^j,V^m\rangle\ge
\kappa_{m,m}\|V^m\|^2
-\sum_{j=1}^m |\kappa_{m,j}|\,
\|V^j\|\,\|V^m\|$.
Note also that
$\langle\LL U^m,V^m\rangle\ge
\kappa_{m,m}\,\kappa_m^*\langle \LL U^m, U^{m}\rangle
-|\kappa_{m,m-1}|
\sqrt{\kappa^*_m \kappa_{m-1}^*}\,|\langle \LL U^m, U^{m-1}\rangle|
$,
where we used $\frac{\beta_m}{1-\beta_m}\le |\kappa_{m,m-1}|\sqrt{\kappa^*_m \kappa_{m-1}^*}$
(which was established under condition A2${}^*$ in \cite[(5.3)]{NK_L2} for $m\ge 3$, while
$\langle \LL U^m, U^{m-1}\rangle=0$
for $m=1,2$ in view of $U^1=0$;
the evaluation leading to the latter hinges
on a very delicate choice of $\bar\rho^*$ in A2).
Combining the above observations
yields
$\langle\delta_t^\alpha U^m+\LL U^m,V^m\rangle\ge
\kappa_{m,m}(W^m)^2
-|\kappa_{m,m-1}|\,W^{m-1}W^m-\sum_{j=2}^m |\kappa_{m,j}|\,
\|V^j\|\,\|V^m\|\ge \kappa_{m,m}(W^m)^2-\sum_{j=1}^m |\kappa_{m,j}|\,
W^j\,W^{m}=W^m\sum_{j=0}^m \kappa_{m,j}
W^j$.
Thus, the inner product of \eqref{semediscr_method} with
$V^m$ yields\vspace{-0.2cm}
$$
W^m\sum_{j=0}^m \kappa_{m,j}
W^j\le
\|f(\cdot, t_m, U^m)\|\,\|V^m\|\le W^m\,\|f(\cdot, t_m, U^m)\|.\vspace{-0.2cm}
$$
Dividing this by $W^m$ and subtracting $\lambda \|U^m\|$, one gets
$\sum_{j=0}^m \kappa_{m,j}
W^j-\lambda \|U^m\|\lesssim (\tau_1/ t_m)^{\gamma+1}$ (where we also used the theorem hypothesis on $f$).
It remains to note that
$\frac{1}{1-\beta_m}\|U^m\|-\frac{\beta_m}{1-\beta_m}\|U^{m-1}\|\le
\|V^m\|\le W^m$  $\forall\,m\ge1$. Hence, an application of~\eqref{star_eq},
immediately yields
the desired  \eqref{semi_stab}.

(ii)
We now return to the original, more general, assumptions, i.e. $U^1=0$ is no longer assumed, while A2 is assumed with
$1\le K\lesssim 1$ (including the case $K:=M\lesssim 1$).
This more general case is easily reduced to
 the case already addressed in part (i)
 as follows.
First, by imitating part (iii) in the proof  of \cite[Lemma~5.1]{NK_L2} (with minor changes for the semilinear case),
 one gets $\|U^j\|_{L_2(\Omega)}\lesssim \tau_1^{\alpha}\lesssim\U^j$ for $j\le K$
 directly from \eqref{semediscr_method}
 (as $K\lesssim 1$).
 Then, by imitating part (ii) in the proof of \cite[Theorem~3.2]{NK_L2},
 $\{U^j\}_{j\le K}$ are eliminated from \eqref{semediscr_method}, after which one can apply the result of part (i) of this proof.
\end{proof}

\renewcommand{\thetheorema}{\ref{theo_semi}${}^*$}
\begin{theorema}\label{theo_star}
Theorem~\ref{theo_semi} remains valid if
the temporal mesh satisfies more general assumptions A2 and A3 in place of $\{t_j=T(j/M)^r\}_{j=0}^M$.
\end{theorema}

\begin{proof}
The existence of a unique solution $U^m\in H^1_0(\Omega)\cap L_\infty(\Omega)$ for $m\ge 1$
follows from $\lambda\tau_j^{\alpha}<\{\Gamma(2-\alpha)\}^{-1}$ $\forall\,j\ge1$
combined with Remark~\ref{rem_A}
exactly as in \cite[Lemma~2.1(i)]{NK_semil_L1}.

Next,
for the error $e^m:= u(\cdot,t_m)-U^m\in H^1_0(\Omega)\cap L_\infty(\Omega)$, using \eqref{problem} and \eqref{semediscr_method},
and the notation
$u^m:=u(\cdot, t_m)$ for the exact solution and
$r^m:=\delta_{t}^{\alpha} u^m-D_t^\alpha u(\cdot, t_m)$ for the truncation error,
one gets
$e^0=0$ and
$\delta_t^\alpha e^m +\LL e^m+F(\cdot, t_m, e^m)=0$ $\forall\,m\ge1$.
Here the nonlinear term
$F(\cdot, t_m, e^m):=f(\cdot, t_m, u^m+e^m)-f(\cdot, t_m, u^m)-r^m$
satisfies
$|F(\cdot, t_m, s)| \le\lambda |s|+|r^m|$ $\forall\,s\in\R$,
so
$\|F(\cdot, t_j, e^j\|_{L_2(\Omega)}-\lambda \| e^j\|_{L_2(\Omega)}
\le \| r^j\|_{L_2(\Omega)}
 \lesssim (\tau_1/ t_j)^{\gamma+1}$,
 where we used a version of the truncation error bound \cite[(4.6)]{NK_L2}
 with $\gamma:=\min\{\alpha,(3-\alpha)/r-1\}$.

Hence, an application of Lemma~\ref{lem_semi_stab} yields
$\|e^j\|_{L_2(\Omega)} \lesssim  \U^j(\tau_1;\gamma)$, and
it remains to show that
$
\U^j=
\ell_\gamma\tau_1 t_j^{\alpha-1} (\tau_1/t_j)^{\min\{0,\,\gamma\}}\simeq {\mathcal E}^j$
from \eqref{L2_semi_error}.
As $\alpha>0$, one concludes that $\min\{0,\,\gamma\}=
\min\{0,\,(3-\alpha)/r-1\}$, while $\ell_\gamma=1$ iff $r\neq 3-\alpha$ (which is equivalent to $\gamma\neq0$).
Now, in view of $\tau_1\simeq M^{-r}$,
if $r>3-\alpha$, one gets
$\U^j=\tau_1 t_j^{\alpha-1}\simeq M^{-r}t_j^{\alpha-1}\simeq {\mathcal E}^j$.
Similarly, if $r<3-\alpha$, one gets
$\U^j= t_j^{\alpha}(\tau_1/t_j)^{(3-\alpha)/r}\simeq M^{\alpha-3}t_j^{\alpha-(3-\alpha)/r}\simeq {\mathcal E}^j$.
Finally, if $r= 3-\alpha$, (similarly to $r>3-\alpha$) one gets
$\U^j=\ell_0\tau_1 t_j^{\alpha-1}\simeq\ell_0 M^{-r}t_j^{\alpha-1}$,
where, in view of Remark~\ref{rem_ell_sharp}, one can use a sharper version of $\ell_0=1+\ln(t_j/\tau_1)$, so again $\U^j\simeq{\mathcal E}^j$.
\end{proof}


\end{document}